\documentclass[12pt]{amsart}
\usepackage{amscd,amssymb,latexsym}
\usepackage{fullpage}
\input{xypic}
\newcommand{\Mdef}[2]{\newcommand{#1}{\relax \ifmmode #2 \else $#2$\fi}}

\newcommand{\sm }{\wedge}

\newcommand{\tensor}{\otimes}

\newcommand{\Hom}{\mathrm{Hom}}

\Mdef{\bhom}{\mathbf{\hat{H}om}}

\Mdef{\Mod}{\mathrm{mod}}

\newtheorem{thm}{Theorem}[section]
\newtheorem{lemma}[thm]{Lemma}

\newtheorem{cor}[thm]{Corollary}

\theoremstyle{definition}

\newtheorem{example}[thm]{Example}

\newtheorem{remark}[thm]{Remark}

\newcommand{\qqed}{\qed \\[1ex]}
\renewenvironment{proof}[1][\hspace*{-.8ex}]{\noindent {\bf Proof #1:\;}}{\qqed}

\Mdef{\PH} {\Phi^H}
\Mdef{\PK} {\Phi^K}
\Mdef{\PL} {\Phi^L}
\Mdef{\PT} {\Phi^{\T}}

\Mdef{\ef}{E{\cF}_+}
\Mdef{\etf}{\widetilde{E}{\cF}}
\Mdef{\eg}{E{G}_+}
\Mdef{\etg}{\tilde{E}{G}}

\Mdef{\infl}{\mathrm{inf}}
\Mdef{\defl}{\mathrm{def}}
\Mdef{\res}{\mathrm{res}}
\Mdef{\ind}{\mathrm{ind}}
\Mdef{\coind}{\mathrm{coind}}

\Mdef{\univ}{\mathcal{U}}

\Mdef{\Fp}{\mathbb{F}_p}
\Mdef{\Zpinfty}{\Z /p^{\infty}}
\Mdef{\Zpadic}{\Z_p^{\wedge}}

\newcommand{\bi}{\begin{itemize}}
\newcommand{\be}{\begin{enumerate}}
\newcommand{\bc}{\begin{center}}
\newcommand{\bd}{\begin{description}}
\newcommand{\ei}{\end{itemize}}
\newcommand{\ee}{\end{enumerate}}
\newcommand{\ec}{\end{center}}
\newcommand{\ed}{\end{description}}

\newcommand{\lra}{\longrightarrow}

\newcommand{\Gspectra}{\mbox{$G$-{\bf spectra}}}

\Mdef{\we}{\mathbf{we}}
\Mdef{\fib}{\mathbf{fib}}
\Mdef{\cof}{\mathbf{cof}}
\Mdef{\BI}{\mathcal{BI}}

\Mdef{\A}{\mathbb{A}}
\Mdef{\B}{\mathbb{B}}
\Mdef{\C}{\mathbb{C}}
\Mdef{\D}{\mathbb{D}}
\Mdef{\E}{\mathbb{E}}
\Mdef{\T}{\mathbb{T}}
\Mdef{\F}{\mathbb{F}}
\Mdef{\G}{\mathbb{G}}
\Mdef{\I}{\mathbb{I}}
\Mdef{\N}{\mathbb{N}}
\Mdef{\Q}{\mathbb{Q}}
\Mdef{\R}{\mathbb{R}}
\Mdef{\bbS}{\mathbb{S}}
\Mdef{\Z}{\mathbb{Z}}

\Mdef{\bA}{\mathbb{A}}
\Mdef{\bB}{\mathbb{B}}
\Mdef{\bC}{\mathbb{C}}
\Mdef{\bD}{\mathbb{D}}
\Mdef{\bE}{\mathbb{E}}
\Mdef{\bF}{\mathbb{F}}
\Mdef{\bG}{\mathbb{G}}
\Mdef{\bH}{\mathbb{H}}
\Mdef{\bI}{\mathbb{I}}
\Mdef{\bJ}{\mathbb{J}}
\Mdef{\bK}{\mathbb{K}}
\Mdef{\bL}{\mathbb{L}}
\Mdef{\bM}{\mathbb{M}}
\Mdef{\bN}{\mathbb{N}}
\Mdef{\bO}{\mathbb{O}}
\Mdef{\bP}{\mathbb{P}}
\Mdef{\bQ}{\mathbb{Q}}
\Mdef{\bR}{\mathbb{R}}
\Mdef{\bS}{\mathbb{S}}
\Mdef{\bT}{\mathbb{T}}
\Mdef{\bU}{\mathbb{U}}
\Mdef{\bV}{\mathbb{V}}
\Mdef{\bW}{\mathbb{W}}
\Mdef{\bX}{\mathbb{X}}
\Mdef{\bY}{\mathbb{Y}}
\Mdef{\bZ}{\mathbb{Z}}

\Mdef{\cA}{\mathcal{A}}
\Mdef{\cB}{\mathcal{B}}
\Mdef{\cC}{\mathcal{C}}
\Mdef{\mcD}{\mathcal{D}} 
\Mdef{\cE}{\mathcal{E}}
\Mdef{\cF}{\mathcal{F}}
\Mdef{\cG}{\mathcal{G}}
\Mdef{\mcH}{\mathcal{H}} 
\Mdef{\cI}{\mathcal{I}}
\Mdef{\cJ}{\mathcal{J}}
\Mdef{\cK}{\mathcal{K}}
\Mdef{\mcL}{\mathcal{L}}

\Mdef{\cM}{\mathcal{M}}
\Mdef{\cN}{\mathcal{N}}
\Mdef{\cO}{\mathcal{O}}
\Mdef{\cP}{\mathcal{P}}
\Mdef{\cQ}{\mathcal{Q}}
\Mdef{\mcR}{\mathcal{R}}
\Mdef{\cS}{\mathcal{S}}
\Mdef{\cT}{\mathcal{T}}
\Mdef{\cU}{\mathcal{U}}
\Mdef{\cV}{\mathcal{V}}
\Mdef{\cW}{\mathcal{W}}
\Mdef{\cX}{\mathcal{X}}
\Mdef{\cY}{\mathcal{Y}}
\Mdef{\cZ}{\mathcal{Z}}

\Mdef{\ca}{\mathcal{a}}
\Mdef{\ct}{\mathcal{t}}

\Mdef{\At}{\tilde{A}}
\Mdef{\Bt}{\tilde{B}}
\Mdef{\Ct}{\tilde{C}}
\Mdef{\Et}{\tilde{E}}
\Mdef{\Ht}{\tilde{H}}
\Mdef{\Kt}{\tilde{K}}
\Mdef{\Lt}{\tilde{L}}
\Mdef{\Mt}{\tilde{M}}
\Mdef{\Nt}{\tilde{N}}
\Mdef{\Pt}{\tilde{P}}

\Mdef{\tA}{\tilde{A}}
\Mdef{\tB}{\tilde{B}}
\Mdef{\tC}{\tilde{C}}
\Mdef{\tE}{\tilde{E}}
\Mdef{\tH}{\tilde{H}}
\Mdef{\tK}{\tilde{K}}
\Mdef{\tL}{\tilde{L}}
\Mdef{\tM}{\tilde{M}}
\Mdef{\tN}{\tilde{N}}
\Mdef{\tP}{\tilde{P}}

\Mdef{\ft}{\tilde{f}}
\Mdef{\xt}{\tilde{x}}
\Mdef{\yt}{\tilde{y}}

\Mdef{\Ab}{\overline{A}}
\Mdef{\Bb}{\overline{B}}
\Mdef{\Cb}{\overline{C}}
\Mdef{\Db}{\overline{D}}
\Mdef{\Eb}{\overline{E}}
\Mdef{\Fb}{\overline{F}}
\Mdef{\Gb}{\overline{G}}
\Mdef{\Hb}{\overline{H}}
\Mdef{\Ib}{\overline{I}}
\Mdef{\Jb}{\overline{J}}
\Mdef{\Kb}{\overline{K}}
\Mdef{\Lb}{\overline{L}}
\Mdef{\Mb}{\overline{M}}
\Mdef{\Nb}{\overline{N}}
\Mdef{\Ob}{\overline{O}}
\Mdef{\Pb}{\overline{P}}
\Mdef{\Qb}{\overline{Q}}
\Mdef{\Rb}{\overline{R}}
\Mdef{\Sb}{\overline{S}}
\Mdef{\Tb}{\overline{T}}
\Mdef{\Ub}{\overline{U}}
\Mdef{\Vb}{\overline{V}}
\Mdef{\Wb}{\overline{W}}
\Mdef{\Xb}{\overline{X}}
\Mdef{\Yb}{\overline{Y}}
\Mdef{\Zb}{\overline{Z}}

\Mdef{\db}{\overline{d}}
\Mdef{\hb}{\overline{h}}
\Mdef{\qb}{\overline{q}}
\Mdef{\rb}{\overline{r}}
\Mdef{\tb}{\overline{t}}
\Mdef{\ub}{\overline{u}}
\Mdef{\vb}{\overline{v}}

\Mdef{\hc}{\hat{c}}
\Mdef{\he}{\hat{e}}
\Mdef{\hf}{\hat{f}}
\Mdef{\hA}{\hat{A}}
\Mdef{\hH}{\hat{H}}
\Mdef{\hJ}{\hat{J}}
\Mdef{\hM}{\hat{M}}
\Mdef{\hP}{\hat{P}}
\Mdef{\hQ}{\hat{Q}}

\Mdef{\thetab}{\overline{\theta}}
\Mdef{\phib}{\overline{\phi}}

\Mdef{\uA}{\underline{A}}
\Mdef{\uB}{\underline{B}}
\Mdef{\uC}{\underline{C}}
\Mdef{\uD}{\underline{D}}

\Mdef{\bolda}{\mathbf{a}}
\Mdef{\boldb}{\mathbf{b}}
\Mdef{\bfD}{\mathbf{D}}

\Mdef{\fm}{\frak{m}}
\Mdef{\fp}{\frak{p}}

\Mdef{\eps}{\epsilon}

\newcommand{\adjointtriple}[5]{
\diagram
#1 \ar[rr]|-{#3} &&
#5  
\llto<1.2ex>^{#4}
\llto<-1.2ex>_{#2} 
\enddiagram}
\newcommand{\ib}{\overline{i}}
\newcommand{\Gbar}{\overline{G}}
\newcommand{\Hbar}{\overline{H}}

\newcommand{\HBG}{H^*(BG)}
\newcommand{\HBH}{H^*(BH)}

\newcommand{\HBGe}{H^*(BG_e)}
\newcommand{\HBHe}{H^*(BH_e)}

\newcommand{\HBGtw}{H^*(BG_e)[\Gbar]}
\newcommand{\HBHtw}{H^*(BH_e)[\Hbar]}

\newcommand{\freeGspectra}{\mbox{free-$G$-spectra}}
\newcommand{\freeHspectra}{\mbox{free-$H$-spectra}}

\renewcommand{\Gspectra}{\mbox{$G$-spectra}}
\newcommand{\Hspectra}{\mbox{$H$-spectra}}

\newcommand{\HBGmod}{\mbox{$\HBG$-modules}}
\newcommand{\HBHmod}{\mbox{$\HBH$-modules}}

\newcommand{\cellCEGmod}{\mbox{cell-$C^*(EG)$-mod-$G$-spectra}}

\newcommand{\cellCBGmod}{\mbox{cell-$C^*(BG)$-mod-spectra}}

\newcommand{\torsHBGmod}{\mbox{torsion-$\HBG$-modules}}
\newcommand{\DGtorsHBGmod}{\mbox{DG-torsion-$\HBG$-modules}}
\newcommand{\torsHBHmod}{\mbox{torsion-$\HBH$-modules}}
\newcommand{\DGtorsHBHmod}{\mbox{DG-torsion-$\HBH$-modules}}

\newcommand{\DGtorsHBGtwmod}{\mbox{DG-torsion-$\HBGtw$-modules}}

\newcommand{\DGtorsHBHtwmod}{\mbox{DG-torsion-$\HBHtw$-modules}}

\newcommand{\modQG}{\mbox{mod-$\Q G$}}
\newcommand{\modQH}{\mbox{mod-$\Q H$}}

\newcommand{\modCG}{\mbox{mod-$C_*(G)$}}

\newcommand{\DGH}{\mathbb{D}(G|H)}
\newcommand{\DGeHe}{\mathbb{D}(G_e|H_e)}
\newcommand{\DRS}{\mathbb{D}(R|S)}
\newcommand{\RAmod}{\mbox{$R[A]$-modules}}
\newcommand{\SBmod}{\mbox{$S[B]$-modules}}

\begin{document}
\title{Algebraic models of change of groups in rational stable
  equivariant homotopy theory}

\author{J.P.C.Greenlees}
\address{School of Mathematics and Statistics, Hicks Building, 
Sheffield S3 7RH. UK.}
\email{j.greenlees@sheffield.ac.uk}
\date{}

\begin{abstract}
Shipley and the author have given an algebraic model 
for free rational $G$-spectra for a compact Lie group $G$  \cite{gfreeq, gfreeq2}.  
 In the present note we describe, at  the level of homotopy categories, the algebraic models for induction, restriction and
coinduction relating free rational $G$-spectra and free rational
$H$-spectra for a subgroup $H$ of $G$.
\end{abstract}

\maketitle

\tableofcontents

\section{Introduction}

\subsection{Change of groups for spectra}
We are concerned with change of groups functors, so we suppose that
$i: H \lra G$ is the inclusion of a subgroup in a compact Lie group
$G$. The restriction functor $i^*: \Gspectra \lra \Hspectra$ has a
left adjoint $i_*$ and a right adjoint $i_!$ defined on $H$-spectra
$Y$ by $i_*(Y)=G_+\sm_HY$ (induction) and $i_!(Y)=F_H(G_+, Y)$
(coinduction). Altogether we have an adjoint triple
$$\adjointtriple{\Gspectra}{i_*}{i^*}{i_!}{\Hspectra},   $$
and this passes to homotopy categories.

There are a number of cases  where categories of rational $G$-spectra have been shown to be equivalent
to algebraic categories, and one may then ask for the algebraic
counterparts of the change of groups functors. The present paper
considers the case of free spectra, where the algebraic models are at
their simplest. This gives considerable insight in the general case,
since  the nature of the models in the more general case is that they
are built  from one contributions at each closed subgroup, each of
which has the character of the free model (see \cite{Kedziorek,
  AGtoral} for examples of this). 

\subsection{The model for free spectra}
The homotopy category of free $G$-spectra is the full subcategory of
spectra $X$ for which the map $EG_+\sm X \lra X$ is an
equivalence. The above adjoint triple passes to free spectra and gives the adjoint
triple that we study here.

For any compact Lie group $G$, there is an algebraic model of free
rational $G$-spectra \cite{gfreeq, gfreeq2}. If $G$ is connected this is 
the category of differential graded (DG) torsion modules over the
polynomial ring $H^*(BG)$ (rational coefficients understood
throughout).  More generally if $G$ has identity component $G_e$ and component group $\Gbar =G/G_e\cong \pi_0(G)$ we note that $\Gbar$ acts via conjugation on $H^*(BG_e)$ and we may form the twisted group ring $\HBGtw$. A torsion $\HBGtw$-module is one which is torsion as a module over $H^*(BG_e)$ (i.e., its localization at any non-maximal prime ideal is trivial). There is a Quillen equivalence \cite{gfreeq2}
$$\freeGspectra/\Q \simeq \DGtorsHBGtwmod ; $$
in fact, a different Quillen equivalence was given in \cite{gfreeq} in the connected case, and the author expects this too to extend to the general case.

\subsection{Contribution}
The purpose of the present note is to identify the algebraic models of
the change of groups functors, at least at the level of the homotopy
category. We expect that a more precise analysis
will show that (with a suitable choice of models) the change of groups Quillen
adjunctions in topology correspond (under suitably chosen Quillen equivalences
with the algebraic models) to the Quillen adjunctions we identify in algebra.

In Theorem \ref{thm:main} we will identify the algebraic model of the
above adjoint triple, but there are two significant lessons along the
way.  Firstly, the right algebraic models are less easy to guess
than the author expected, even at the crudest
level. Secondly, we show that the algebraic models depend on which of the
two Quillen equivalences (i.e., from  \cite{gfreeq} or from \cite{gfreeq2}) are used. 

 The point of writing this elementary note separately is that 
the proofs are simple formal manipulations whilst the conclusions 
significantly illuminate the wider landscape. Of course one expects
this to be helpful in lifting  these results to the model categorical
level, but the usefulness goes beyond that. 
For example, it was helpful in guiding the search for certain functors
in \cite{Kedziorek} and in identifying the algebraic model for toral
spectra  \cite{AGtoral}. The author found it illuminated phenomena
from \cite{BGca, kappaI, BGstrat} and that it was a useful guide in dealing
with models of categories of modules over profinite groups. Finally,
it explained a delicate point in comparing the results of 
\cite{gfreeq} and \cite{gfreeq2}.

\subsection{A mismatch}
A moment's thought shows that it is not straightforward to write down
the models. Suppose for the present that $G$ and $H$ are connected and
write 
$$\theta =i^*: \HBG \lra \HBH$$
for the induced map in cohomology. The ring homomorphism $\theta$ gives a restriction functor $\theta^*: \HBHmod \lra \HBGmod$ with left adjoint $\theta_*$ and right adjoint $\theta_{!}$ given on $\HBG$-modules $M$ by $\theta_*( M )=\HBH \tensor_{\HBG} M$ (extension of scalars) and 
$\theta_*( M )=\Hom_{\HBG}(\HBH , M)$ (coextension of scalars), altogether giving
$$\adjointtriple{\HBHmod}{\theta_*}{\theta^*}{\theta_!}{\HBGmod}, $$
and these functors pass to torsion modules and again give an adjoint triple. 

We note that in topology two of these functors go from $H$-spectra to
$G$-spectra, whereas in algebra two functors go from $\HBG$-modules to
$\HBH$-modules. It is clear these two triples cannot correspond to
each other, and that we must find other functors in algebra to model
the functors in topology. Our technique in the various cases is to choose the most accessible of the three functors $i_*, i^* $ and $i_!$ and find
its algebraic model. The models of the other two follow by taking adjoints.

\subsection{Conventions}
This  paper is a contribution to rational equivariant stable homotopy
theory, so we are concerned with algebraic models of categories of rational
equivariant spectra. We work throughout at the level of homotopy
categories and our conclusions are independent of models provided well-behaved
change of groups functors exist. In particular our results apply to
orthogonal spectra \cite{MM} and to the spectra of  \cite{lmsm}. Given
a choice of model for $G$-spectra, there are numerous models
for free $G$-spectra, and we choose to localize with respect to
non-equivariant homotopy. Similarly, models for rational spectra are
obtained by localizing with respect to rational homotopy.

Since the algebraic model is only available rationally, our notation
will leave rationalization implicit: all algebras are assumed to be $\Q$-algebras and all cohomology
 has coefficients in $\Q$.

\section{The case of finite groups}
We suppose first that $G$ and $H$ are finite. There are two obvious equivalences 
$$\freeGspectra \simeq \mbox{DG-$\modQG$}, $$
given by hom or tensor Morita equivalences. We are using the fact that we have an equivalence  $\Q G\simeq F(G_+, G_+)^G$ identifying the endomorphism ring of the generator $G_+$ with  an Eilenberg-MacLane spectrum. 
Now we may use the hom Morita equivalence and let a $G$-spectrum $X$ be mapped to the right $\Q G$-module
$E_GX:=F(G_+, X)^G$, or we may use the tensor Morita equivalence in which it is mapped to the left $\Q G$-module 
$G_+\sm_G X$, which is made into a right $\Q G$-module in the usual way. These are naturally equivalent.

Now the inclusion $i: H \lra G$ induces a map  $\ib : \Q H \lra \Q G$ of group rings, and hence restriction $\ib^*: \modQG \lra \modQH$ with left adjoint $\ib_*$ induction and right adjoint coinduction as usual. We observe
$$E_Hi^*(X) = F(H_+,i^*X)^H=\ib^*F(G_+,X)^G=\ib^*E_GX. $$

We conclude that in this case, the spectrum level restrictions $i_*\vdash i^*\vdash i_!$ correspond to the functors
$\ib_*\vdash \ib^*\vdash \ib_!$ just as expected.   

For finite groups, we note that this applies more generally; the group homomorphism $i: H \lra G$ need not be a monomorphism. We may pull back a free $G$-spectrum to obtain an $H$-spectrum. If $i$ is not a monomorphism 
the resulting spectrum need not be free, but we can make it so, and the functor $i^*: \freeGspectra \lra \freeHspectra$
is induced by restriction along the homomorphism $\ib: \Q H \lra \Q
G$. The left and right adjoint  of $\ib^*$ are extension and
coextension of scalars, and by Maschke's Theorem these are
automatically homotopy invariant.

\section{The case of connected groups}
Now suppose that $G$ and $H$ are connected. It turns out that the
equivalence from \cite{gfreeq} and \cite{gfreeq2} give {\em different}
algebraic models for change of groups functors. We will describe three
equivalences between topology and algebra, the first two 
give the same models for change of groups, and the third (namely that of \cite{gfreeq}) gives a different model. 

\subsection{Eilenberg-Moore correspondents}
\label{subsec:EM}
This subsection discusses the equivalence of \cite{gfreeq2}. 

In this case, the equivalence 
$$\freeGspectra \simeq \DGtorsHBGmod$$
 is a composite of two equivalences. The first is a change of rings
 along $\bbS \lra  DEG_+=: C^*(EG)$, where $\bbS$ is the sphere
 spectrum. Since $G$-spectra are simply $\bbS$-modules, and since the
 change of rings map is a non-equivariant equivalence, 
 cellularization with respect to $G_+$ gives a Quillen equivalence 
$\freeGspectra \simeq \cellCEGmod$ (see \cite[Section 4]{cellprin} for
full details). In the  second we pass to fixed points to give
$\cellCEGmod \simeq \cellCBGmod$, since $C^*(EG)^G\simeq C^*(BG)$
\cite[Section 8]{modulefps}. Since we are working rationally,
$C^*(BG)$ is formal and  we will work with $\HBG$ rather than $C^*(BG) $-throughout. 

In the first we have only changed the ring we are working over, so
this plays no direct role in the change of groups. In the second equivalence we need to make a short manipulation. 

\begin{lemma}
\label{lem:EMcorr}
$i_! $ corresponds to $\theta^*$. 
\end{lemma}

\begin{proof}
We calculate
$$(i_! (DEH_+ \sm X))^G=F_H(G_+, DEH_+ \sm Y)^G=i^*(DEH_+ \sm Y)^H. $$
\end{proof}

We will infer the correspondents of $i^*$ and $i_*$ in Subsection \ref{subsec:EMcorr} below. 

\subsection{Koszul correspondents I}
\label{subsec:KoszulI}
This case is closely related to that of \cite{gfreeq} described in Subsection \ref{subsec:KoszulII}, but significantly different. 

In this case the equivalence 
$$\freeGspectra \simeq \DGtorsHBGmod$$ 
is again a composite of two equivalences. The first is  the Morita
equivalence  $\freeGspectra \simeq \modCG$ \cite{tec,SS}, where $C_*(G)$ is now the
group ring spectrum  of $G$ over the rational sphere spectrum
(literally $\Q \sm G_+$, so that we could have continued with the
notation $\Q G$ from the case of finite groups); the equivalence is given again by taking a $G$-spectrum $X$ to the right $C_*(G)$-module $E_G(X) =F(G_+, X)^G$. The second equivalence is between modules over $C_*(G)$ and modules over $C^*(BG)=F(BG_+, \Q)$; this is another Morita equivalence, where a $C_*(G)$-module  $P$ is taken to $E'_G(P)=\Hom_{C_*(G)}(\Q , P)$, where we note that 
this is a module over $\Hom_{C_*(G)}(\Q, \Q)=\Hom_{C_*(G)}(C_*(EG), C_*(EG))\simeq C^*(BG)$. 

For the first step, it follows as in the finite case that $i^*$ correpondes to $\ib^*$. The more interesting step is the second. 

\begin{lemma}
\label{lem:KoszulIcorr}
$i_!$ corresponds to $\theta^*$. 
\end{lemma}

\begin{proof}
Since $i_!$ corresponds to $\ib_!$ it suffices to deal with modules over the chains on groups.   
For a $C_*(H)$-module $Q$, we  calculate 
$$E'_G(\ib_!Q) =\Hom_{C_*(G)}(\Q , \ib_! Q)=\theta^* \Hom_{C_*(H)}(\Q, Q). $$
\end{proof}

We will infer the correspondents of $i^*$ and $i_*$ in Subsection \ref{subsec:EMcorr} below. 

\subsection{Koszul correspondents II}
\label{subsec:KoszulII}
This case was treated in \cite{gfreeq}, but we recap here to show the difference from the case of Subsection 
\ref{subsec:KoszulI}.

In this case the equivalence 
$$\freeGspectra \simeq \DGtorsHBGmod$$ 
is a again a composite of two equivalences. The first is   the Morita equivalence  $\freeGspectra \simeq \modCG$,
precisely as in Subsection \ref{subsec:KoszulI} and it follows as before that $i^*$ correponds to $\ib^*$. 

However for the second equivalence we use the reverse Morita
equivalence \cite{tec,SS}.  This time, a 
right $C_*(G)$-module  $P$ is taken to $T_G(P)=P \tensor_{C_*(G)} \Q$,  where we note that 
this is a left module over $\Hom_{C_*(G)}(\Q, \Q)=\Hom_{C_*(G)}(C_*(EG), C_*(EG))\simeq C^*(BG)$.

\begin{lemma}
\label{lem:KoszulIIcorr}
$i_*$ corresponds to $\theta^*$. 
\end{lemma}

\begin{proof}
For a $C_*(H)$-module $Q$, we  calculate 
$$T_G(\ib_*Q) =(Q\tensor_{C_*(H)}C_*(G))\tensor_{C_*(G)}\Q=\theta^* (Q\tensor_{C_*(H)}\Q). $$
\end{proof}

We will infer the correspondents of $i^*$ and $i_!$ in Subsection \ref{subsec:KoszulIIcorr} below.
 
\subsection{Iterated adjoints of restriction of scalars}
We consider a map $\theta : R\lra S$ of commutative rings. The example we have in mind is 
$R=\HBG, S=\HBH$ and $\theta =i^*$. 

 The ring homomorphism induces a restriction of scalars $\theta^*$, which always has a left adjoint 
$\theta_*$ defined by $\theta_*(M)=S\tensor_R M$ and a right adjoint $\theta_!$ defined by 
$\theta_!(M)=\Hom_R(S, M)$.  If $S$ is small as an $R$-module there are further adjoints. 
To describe these,  we first define the {\em relative dualizing complex} 
$$\DRS =\Hom_R(S,R) , $$
which we note is an $(R,S)$-bimodule. 

\begin{lemma}
\label{lem:adjstring}
If $S$ is small as an $R$ module we have a string of adjunctions
$$\theta^{\dagger}\vdash \theta_* \vdash \theta^* \vdash \theta_! \vdash \theta^!$$
defined on $R$-modules $M$ and $S$-modules $N$  by 
\begin{itemize}
\item $\theta^{\dagger} (N)=\DRS \tensor_S N$
\item $\theta_*(M)=S\tensor_RM$
\item $\theta^*(N)$ is $N$ with restricted action
\item $\theta_!(M)=\Hom_R(S,M)$
\item $\theta^!(N)=\Hom_S(\DRS ,N)$
\end{itemize}
\end{lemma}

\begin{remark}
The 2014 PhD thesis of M.Abbasirad \cite{abbasirad} investigates the model categorical underpinnings of this.
\end{remark}

\begin{proof}
We write  $DM=\Hom_R(M,R)$ for the Spanier-Whitehead duality functor, and note that when 
$M'$ is small the natural maps $M'\lra DDM'$ and $DM'\tensor_R M \lra \Hom_R(M',M)$ are equivalences. 
If $S$ is small over $R$ then the case $M'=S$ gives  
$$\theta_*(M)=S\tensor_RM \simeq \Hom_R(DS, R)\tensor_R M \simeq \Hom_R(DS, M)$$
and 
$$\theta_!(M)=\Hom_R(S, M)\simeq DS \tensor_R M.$$

Since $\DRS =DS$, the results follow from the usual Hom-tensor adjunction. 
\end{proof}

The relevant special case is as follows. 

\begin{cor}
With $\theta =i^*: \HBG \lra \HBH$, the relative dualizing module is
given by 
$$\DGH=\Hom_{\HBG}(\HBH, \HBG)\simeq \Sigma^{L(G/H)}H^*(BH).$$
Accordingly, the adjoint functors
$$\theta^{\dagger}\vdash \theta_* \vdash \theta^* \vdash \theta_! \vdash \theta^!$$
are defined on $\HBG$-modules $M$ and $\HBH$-modules $N$  by 
\begin{itemize}
\item $\theta^{\dagger} (N)=\Sigma^{LG/H} N$
\item $\theta_*(M)=\HBH\tensor_{\HBG} M$
\item $\theta^*(N)$ is $N$ with restricted action
\item $\theta_!(M)=\Hom_{\HBG}(\HBH,M)$
\item $\theta^!(N)=\Sigma^{-L(G/H)} N$.
\end{itemize}
\end{cor}

\begin{remark}
If we only take into account the module structures over $\HBG$ and
$\HBH$, then the suspension by $L(G/H)$ is simply the integer suspension by $\dim (G/H)$ because
$G$ and $H$ are connected. We have written it using tangent spaces to
record the functoriality in $G$ and $H$. This will be important when
we treat the case of disconnected groups. 
\end{remark}

\begin{proof}
Note that $\HBG$ is a polynomial ring and by  Venkov's theorem $\HBH$ is a finitely generated
$\HBG$-module, so the condition of Lemma \ref{lem:adjstring} is
satisfied. 

The identification of the dualizing module is given in \cite[Theorem 6.8]{BGstrat}
using stable equivariant homotopy. 
\end{proof}

\subsection{The Eilenberg-Moore correspondents of \cite{gfreeq2}}
\label{subsec:EMcorr}

We showed in Lemma \ref{lem:EMcorr} above that $i_!$ corresponds to $\theta^*$. 
It follows that the left adjoints of $i_!$ and $\theta^*$ correspond, so that $i^*$ corresponds to $\theta_*$. Taking left adjoints again, it follows that  $i_*$ corresponds to the left adjoint of $\theta_*$. 

Altogether we find
$$\adjointtriple{\freeGspectra}{i_*}{i^*}{i_!}{\freeHspectra},  $$
corresponds to
$$\adjointtriple{\torsHBGmod}{\theta^!}{\theta_*}{\theta^*}{\torsHBHmod}.$$

Of course this applies equally to the Koszul I equivalence described in Subsection \ref{subsec:KoszulI} using 
Lemma \ref{lem:KoszulIcorr}. 

\subsection{The Koszul II correspondents of \cite{gfreeq}}
\label{subsec:KoszulIIcorr}

We showed in Lemma \ref{lem:KoszulIIcorr} that $i_*$ corresponds to $\theta^*$. 
It follows that the left and right adjoints of $i_*$ correspond to the left and right adjoints of 
 $\theta^*$.

Altogether we find
$$\adjointtriple{\freeGspectra}{i_*}{i^*}{i_!}{\freeHspectra},  $$
corresponds to
$$\adjointtriple{\DGtorsHBGmod}{\theta^*}{\theta_!}{\theta^!}{\DGtorsHBHmod}.$$

\section{Twisted group rings}
\label{sec:twisted}
Before turning to the case of arbitrary compact Lie groups, it is worth recording some general facts about
the algebra of twisted group rings. 

First suppose $R$ is a commutative $\Q$-algebra with an action of a finite group $A$, so that we may form the twisted
group ring $R[A]$. If $R$ is a $k$-algebra, and $M_1, M_2$ are $R[A]$-modules
$$\Hom_{R[A]} (M_1, M_2) =\Hom_{R} (M_1, M_2)^A$$

Now suppose that $\ib : B\lra A$ is a homomorphism of finite groups. We want
to observe that extension and coextension of scalars along $\ib: \Q B
\lra \Q A$ (which are group theoretic induction and coinduction when
$\ib$ is a monomorphism) are compatible with the action on $R$. 
This is straightforward once we observe we  have  a map $\ib: R[B] \lra R[A]$. 

\begin{lemma}
Suppose $N$ is an $R[B]$-module. 
\begin{itemize}
\item The natural map is an isomorphism $\Q A\tensor_{\Q B}N\cong R A\tensor_{R
   B}N$, so that the two possible meanings of $\ib_*N$ are
 compatible. 
\item The natural map is  an isomorphism $\Hom_{\Q B}(\Q A, N)\cong 
\Hom_{R B}(R A, N)$,  so that the two possible meanings of $\ib_!N$ are
 compatible. 
\item 
We have  an adjoint triple
$$\adjointtriple{\RAmod}{\ib_*}{\ib^*}{\ib_!}{\SBmod}. \qqed$$
\end{itemize} 
\end{lemma}


Now suppose given a map $\theta_e : R\lra S$ of commutative rings. 
We suppose given an action of a finite group $A$ on $R$ and a finite group $B $ on $S$, 
and that these are compatible via a  homomorphism 
$\ib: B\lra A$ of finite groups in the sense that   $\theta_e (\ib (b)
r)=b \theta_e (r)$. We may then write $\theta =(\theta_e, \ib)$ for
the pair of structure maps. 
   

\begin{lemma}
We have an adjoint triple
$$\adjointtriple{\RAmod}{\theta^{\dagger}}{\theta_*}{\theta^*}{\SBmod}.$$
The functors are defined as follows.
\begin{itemize}
\item  The functor $\theta_*=\theta^e_* : \RAmod \lra \SBmod$ is defined by 
the formula  $\theta_*(M)=S\tensor_R M$, where the $B$-action is diagonal, using $\ib$ on the second factor
(so we might write $\theta_*=\ib^*\theta^e_*$).  
\item The left adjoint of $\theta_*$ is the functor $\theta^{\dagger}$ defined by 
 $\theta^{\dagger}(N)=\ib_*\theta_e^{\dagger}N$. 
\item The right adjoint of $\theta_*$ is the functor $\theta^*$ defined by 
 $\theta^*(N)=\ib_!\theta_e^* N$. 
\end{itemize}
\end{lemma}

\begin{proof}
In view of the fact that the space of $R[A]$-maps is just the $A$-equivariant $R$-maps, and similarly for $S[B]$, 
this is just a case of checking that the adjunctions defined above are equivariant. 
\end{proof}

\section{The general case}
\label{sec:general}
For a general compact Lie group $G$ and subgroup $H$, we have a diagram 
$$\diagram
1\rto &G_e \rto & G\rto &\Gbar \rto &1\\
1\rto &H_e \uto^{i_e} \rto & H \uto^i \rto &\Hbar \uto^{\ib} \rto &1
\enddiagram$$
We note that $\ib$ need not be a monomorphism. 

We observe that with $R=\HBGe$, $A=\Gbar$,  $S=\HBHe$, $B=\Hbar$ and
$\theta_e =i^* : \HBGe \lra \HBHe$ and $\theta =(\theta_e, \ib)$,  
we obtain an instance of the twisted group ring context described in
the Section \ref{sec:twisted}. 
The compatibility of $\theta_e$ and $\ib$ arises since the action of $B=\Hbar$ on $H_e$ by conjugation 
is compatible with the action of $A=\Gbar$ by conjugation on $G_e$, and the action of $G_e$ on $\HBGe$ is  trivial. 

The first stage of the equivalence 
$$\freeGspectra \simeq \DGtorsHBGtwmod$$
 is again  extension of scalars along $\bbS \lra DEG_+=:C^*(EG)$. However the second stage is now to take
$G_e$ fixed points (rather than $G$-fixed points) and to consider the value in free 
$\Gbar$-spectra. 

\begin{lemma}
$i_!$ corresponds to $\theta^*$. 
\end{lemma}

\begin{proof}
The invariant of an $H$-spectrum $Y$ is the $\Hb$-spectrum $(Y\sm
DEH_+)^{H_e}$ (or rather its homotopy groups as an $\Hb$-module. The
invariant of the $G$-spectrum $i_!Y$ is the $\Gb$-spectrum 
$$(F_H(G_+, Y)\sm DEG_+)^{G_e}= F_H(G_+, Y\sm DEH_+)^{G_e}  $$
(or rather its homotopy groups as a $\Gb$-module). 
Now if $T$ is inflated from $\Gbar$ we have
$$[T, F_H(G_+, Y')^{G_e}]^{\Gbar}=[T, F_H(G_+, Y')]^{G} =[T, Y']^{H} =[T, (Y')^{H_e}]^{\Hbar}$$
as required.  
%
\end{proof}

This allows us to give the general case. 

\begin{thm}
\label{thm:main}
Under the equivalence of \cite{gfreeq2}, the adjoint triple
$$\adjointtriple{\freeGspectra}{i_*}{i^*}{i_!}{\freeHspectra},  $$
corresponds to the adjoint triple
$$\adjointtriple{\DGtorsHBGtwmod}{\theta^{\dagger}}{\theta_*}{\theta^*}{\DGtorsHBHtwmod}.$$
These algebraic models of the change of groups morphisms are given
explicitly by 
\begin{itemize} 
\item $$\theta^{\dagger}(N)=\Q [\Gbar]\tensor_{\Q [\Hbar ]} \Sigma^{LG/H}N, $$
\item $$\theta_*(M)=\HBHe \tensor_{\HBGe}M$$
\item $$\theta^*(N)=\Hom_{\Q [\Hbar]} (\Q [\Gbar], N)$$
\end{itemize}
where  $\Hb$ acts on $L(G/H)$ by differentiating the conjugation
action. \qqed 
\end{thm}

\begin{example}
We take $G=SO(3)$, $H=O(2)$. Now $\HBG =H^*(BSO(3))=\Q [d]$, 
$\HBHe=H^*(BSO(2))=\Q[c]$ with $\Hbar =W$ of order 2, acting on $\Q [c]$ 
by $c\mapsto -c$. The map $\ib : \Q W \lra \Q$ is the augmentation. The inclusion of identity components induces 
$\theta_e: \Q [d]\lra \Q [c]$ which is given by mapping $d $ to $c^2$.
 Remembering that cohomological degrees are negative, the relative dualizing module is given by 
$$\DGeHe =\Hom_{\Q [d]}(\Q[c], \Q[d])\cong  c^{-1}\cdot \Q [c]\cong
\Sigma^2 \tilde{\Q}\tensor \Q [c].$$

Free $SO(3)$-spectra are modelled by torsion $\Q[d]$-modules, and
free $O(2)$-spectra are modelled by torsion $\Q[c][W]$-modules. The forgetful functor
$i^*$ is modelled by 
$$\theta_*M=\Q [c] \tensor_{\Q [d]} M, $$
the induction functor is modelled by 
$$\theta^{\dagger} (N)=\Sigma^2(\tilde{\Q}\tensor N)_W$$
and  the coinduction functor is modelled by 
$$\theta^* (N)=N^W. \qqed $$
\end{example}


\begin{thebibliography}{999}
\bibitem{abbasirad} M. Abbasirad
``Homotopy theory of differential graded modules and adjoints of restrictions of scalars.''
 Sheffield PhD Thesis (2014) 122pp
\bibitem{BGca}
  D.J.Benson and J.P.C.Greenlees 
  ``Commutative algebra for cohomology
  rings of classifying spaces of compact Lie groups.'' 
  J. Pure and Applied Algebra 
  {\bf 122} (1997) 41-53.
\bibitem{kappaI}
   D.J.Benson and J.P.C.Greenlees
   ``Localization and duality in topology and  modular representation theory.''
   JPAA {\bf 212} (2008) 1716-1743
\bibitem{BGstrat}
D.J.Benson and J.P.C.Greenlees
``Stratifying the derived category of cochains on $BG$ for $G$ a
compact Lie group.''
JPAA {\bf 218} (2014), 642-650
\bibitem{tec} 
   W.G.Dwyer and J.P.C.Greenlees 
   ``Complete modules and torsion modules.''
   American J. Math. {\bf 124} (2002) 199-220
\bibitem{o2q} J.P.C.Greenlees 
  ``Rational $O(2)$-equivariant cohomology theories.'' 
  Fields Institute Communications {\bf 19} (1998) 103-110
\bibitem{s1q} J.P.C.Greenlees 
  ``Rational $S^1$-equivariant stable homotopy theory.''
  Mem. American Math. Soc. (1999) xii+289 pp.
\bibitem{so3q} J.P.C.Greenlees 
  ``Rational $SO(3)$-equivariant cohomology theories.'' 
   Contemporary Maths. {\bf 271}, American Math. Soc.
   (2001) 99-125
\bibitem{AGtoral} 
 J.P.C.Greenlees 
  ``Rational equivariant cohomology theories with toral support''
 Preprint (2015) 42pp arXiv 1501.03425 
\bibitem{gfreeq}
J.P.C.Greenlees and B.E.Shipley 
``An algebraic model for free rational $G$-spectra for compact connected
Lie groups $G$.''
 Math Z {\bf 269} (2011) 373-400, DOI 10.1007/s00209-010-0741-2
\bibitem{gfreeq2}
J.P.C.Greenlees and B.E.Shipley 
``An algebraic model for free rational $G$-spectra.''
Bull. LMS {\bf 46} (2014) 133-142, DOI 10.1112/blms/bdt066
\bibitem{cellprin}
J.P.C.Greenlees and B.E.Shipley 
``The cellularization principle''
HHA {\bf 15} (2013) 173-184, arXiv:1301.5583
\bibitem{modulefps}
J.P.C.Greenlees and B.E.Shipley
``Fixed point adjunctions for module spectra.''
Algebraic and Geometric Topology 14 (2014) 1779-1799
arXiv:1301.5869
\bibitem{Kedziorek} M.Kedziorek 
``Algebraic models for rational $G$-spectra.''
Thesis (2014) 176pp
\bibitem{lmsm} L.G.Lewis, J.P.May and M.Steinberger (with contributions
by J.E.McClure) ``Equivariant stable homotopy theory''
Lecture notes in mathematics {\bf 1213}, Springer-Verlag (1986) 
\bibitem{MM}
     M.~Mandell and J.~P.~May, 
     {\em Equivariant Orthogonal Spectra and S-Modules},
     Mem. Amer. Math. Soc. {\bf 159} (2002), no. 755, x+108 pp.
\bibitem{SS} S.Schwede and B.E.Shipley 
   ``Stable model categories are categories of modules.''
    Topology {\bf 42} (2003), no. 1, 103--153.
\end{thebibliography}
\end{document}